\documentclass[12pt,a4paper]{article}
\usepackage[utf8]{inputenc}
\usepackage[T1]{fontenc}
\usepackage[english]{babel}
\usepackage{amssymb,color,wrapfig, float}
\usepackage{mathtools}
\usepackage{indentfirst}
\usepackage{graphicx}
\usepackage{amsmath}
\usepackage{url}
\usepackage{indentfirst}
\usepackage{misccorr}
\usepackage{amsmath,amssymb,url,xcolor}
\usepackage{scrextend}
\usepackage[T2A]{fontenc}
\usepackage{mathtools}
\usepackage{amsthm}
\usepackage[colorlinks=true,linkcolor=blue,urlcolor=red,unicode=true,hyperfootnotes=false,bookmarksnumbered]{hyperref}

\usepackage[letterpaper,top=2cm,bottom=2cm,left=3cm,right=3cm,marginparwidth=1.75cm]{geometry}

\newcommand{\M}{\mathcal{M}}
\newcommand{\T}{\mathcal{T}}

\DeclareMathOperator{\supp}{supp}

\newtheorem{theorem}{Theorem}
\newtheorem{lemma}[theorem]{Lemma}

\newtheorem{definition}{Definition}
\newtheorem{claim}[theorem]{Claim}

\title{On the maximum number of vectors in $\{0,\pm1\}^n$ with forbidden inner products}

\author{Ilya Lobatskii\thanks{Moscow Institute of Physics and Technology; E-mail: \url{ lobzik22848@gmail.com}}, Yakov Shubin\thanks{Moscow Institute of Physics and Technology;
E-mail: \url{shubin.yakoff@gmail.com}}}

\begin{document}
\maketitle
\nocite{*}

\newenvironment{Proof}
{\par\noindent{\bf Proof.}}
{\hfill$\scriptstyle\blacksquare$}

\begin{abstract}
    Let $M \subset \{0,\pm1\}^n$ be a set such that $(m,m)=4$ for every $m\in M$, and $(m_1,m_2)\in\{-4,-3,-2,-1,0,3\}$ for any two distinct vectors $m_1,m_2\in M$. We determine the maximum possible cardinality of such a set $M$ for all sufficiently large $n$.
\end{abstract}

\textbf{Keywords:} extremal set theory, vectors in $\{0,\pm1\}^n$.

\section{Introduction}

In this paper, we consider the problem of finding the maximum number of vectors in $\{0,\pm 1\}^n$ with a fixed number of nonzero coordinates and whose pairwise inner products take values in a set $L$. We first recall the relevant background.

A substantial body of work has been devoted to the analogous problem for vectors in $\{0,1\}^n$. In this setting, vectors can be identified with subsets of an $n$-element set, while the values of inner products can be identified with the cardinalities of intersections of the corresponding subsets. One of the fundamental results of this kind is the Erd\H{o}s--Ko--Rado theorem (see \cite{EKR}), which determines the largest number of vectors in $\{0,1\}^n$ with $k$ nonzero coordinates and pairwise positive inner products, that is, with the zero inner product forbidden. This theorem has became the foundation of an entire area of extremal set theory.

A related problem, usually referred to as the Erd\H{o}s--S\'os problem, asks for the maximum size in the case where exactly one value of the inner product, denoted by $t$, is forbidden; see \cite{ES}. This problem remains open in its general form. In 1985, Frankl and F\"uredi proved a breakthrough result, solving it in the case $n > n_0(k,t)$, $k > 2t$ (see \cite{FF}). In recent years, a number of results have also been obtained for some special cases of the Erd\H{o}s--S\'os problem in \cite{Linz}, \cite{KZ}, \cite{EKL}, and \cite{los}.

All such problems can be formulated in terms of finding the independence number of a certain graph. Recall that a subset of vertices of a graph $G$ is called independent if no two of its vertices are joined by an edge. The size of the largest independent subset in a graph $G$ is called its \textit{independence number} and is denoted by $\alpha(G)$. Define the graph $G(n,k,t)$ whose vertices are all possible $k$-element subsets of the set $[n]=\{1,2, \ldots, n\}$, with an edge between two vertices whenever the intersection of the corresponding subsets has exactly $t$ elements.

The graphs $G(n,k,t)$ are important in extremal set theory, combinatorial geometry, and the study of codes with forbidden distances. We give several examples. In 1981, Frankl and Wilson used an estimate for the independence number of the graph $G(4p,2p,p)$ to show that the chromatic number of Euclidean space grows exponentially with the dimension (see \cite{FW}). In 1991, Kahn and Kalai used the results of Frankl and Wilson to disprove the classical Borsuk conjecture that every bounded set in $\mathbb{R}^{n}$ with more than one point can be partitioned into $n+1$ parts of smaller diameter (see \cite{KK}, \cite{RB}, \cite{RB2}). A more detailed survey can be found in \cite{Gnkt}.

We now pass to a more general formulation in which more than one inner product may be forbidden. Deza, Erd\H{o}s, and Frankl~\cite{DEF} initiated the study of $(n,k,L)$-systems, that is, families of $k$-element subsets of an $n$-element set in which all pairwise intersection sizes belong to the set $L$. It is easy to see that the Erd\H{o}s--S\'os problem is a special case of this problem for $L=\{0,1,\ldots,k\} \setminus \{t\}$. Let $m(n,k,L)$ denote the maximum possible number of sets in an $(n,k,L)$-system (see \cite{KupSur}). The problem of determining $m(n,k,L)$ in general appears to be extremely difficult. A number of papers address various special cases of this problem (for example, see \cite{AK}, \cite{DES}, \cite{Fur}). For instance, in \cite{Fr}, Frankl found the order of growth of $m(n,k,L)$ for all $k \leqslant 7$ and all values of $L$, except for the cases $k=7$, $L=\{0,2,3,5\}$ and $L=\{0,2,3,5,6\}$.

We now turn to problems concerning sets of vectors in $\{0,\pm 1\}^n$ with forbidden inner products. Several papers have been devoted to the study of such sets (for example, see \cite{CK}, \cite{PR}, \cite{ABR}, \cite{LR}, \cite{FK}, \cite{FK2}). For such sets, there is also an analogue of the Erd\H{o}s--S\'os problem: one seeks the maximum number of vectors in $\{0,\pm 1\}^n$ with $k$ nonzero coordinates in which one value $t$ of the pairwise inner products is forbidden. This problem was solved in \cite{CK} for $t=0$, and also for odd negative $t$ in the regime $n>n_0(k,t)$. For negative even $t$, the exact answer is unknown in general: Frankl and Kupavskii obtained \cite{FKanti}, an asymptotic result was obtained for all negative $t$.

Passing from vectors in $\{0,1\}^n$ to vectors in $\{0,\pm 1\}^n$ made it possible to improve the lower bound in the well-known Nelson--Erd\H{o}s--Hadwiger problem (see \cite{Rchrom}) and also in Borsuk's problem (see \cite{RB3}).

The following theorem was proved in \cite{CKR}.

\begin{theorem}\label{theoremCKR}
    Let $n$ be a positive integer. Define $c(n)$ as follows:
    $$ c(n)=\left\{\begin{array}{ll}
0 & \text { if } n \equiv 0 \\
1 & \text { if } n \equiv 1 \\
2 & \text { if } n \equiv 2 \text { or } 3
\end{array} \quad(\bmod \ 4) .\right.
$$
Let $\mathcal{V} \subset \{0,\pm 1\}^n$ be a set such that $(v,v)=3$ for every $v \in \mathcal{V}$, and $(v_1,v_2) \neq 0$ for any two distinct vectors $v_1,v_2 \in \mathcal{V}$. Then
$$|\mathcal{V}| \leqslant \max\{6n-28,4n-4c(n)\}.$$
\end{theorem}

Theorem~\ref{theoremCKR} is analogous to a simpler theorem of Nagy~\cite{Na} for vectors in $\{0,1\}^n$ with inner product 1 forbidden. The proof of Theorem~\ref{theoremCKR} in the case $n \geqslant 14$ is based on the analysis of certain structures called snakes and cobras. For the cases $n<14$, Cherkashin, Kulikov, and Raigorodskii used a computer search, treating each value separately. The reason is that, for small values of $n$, the extremal example may be different, and there may also be several such examples. In the present paper we focus on the method of snakes, but for a more complicated problem; therefore, we assume $n$ is sufficiently large.

We consider families of vectors in $\{0,\pm 1\}^n$ with exactly 4 nonzero coordinates for which the inner products 1 and 2 are forbidden. It is worth noting that the forbidden values 1 and 2 are not arbitrary. Sets of vectors with 4 nonzero coordinates and these forbidden inner products were used in \cite{RS} as a counterexample to Borsuk's conjecture for $\ell_p$ metrics.

The main result of this paper is the following theorem.

\begin{theorem}\label{bigtheorem}
    Let $\mathcal{M} \subset \{0,\pm 1\}^n$ be a set such that $(m,m)=4$ for every $m \in \mathcal{M}$, and $(m_1,m_2) \in \{-4, -3, -2, -1, 0, 3\}$ for any two distinct vectors $m_1,m_2 \in \mathcal{M}$. Then there exists $n_0$ such that for all $n \geqslant n_0$ the maximum possible cardinality of $\mathcal{M}$ is equal to $8n-48$.
\end{theorem}


The rest of the paper is organized as follows. In Section 2 we introduce the definitions needed for the proof and formulate several simple statements. In Section 3 we prove the lower bound in Theorem~\ref{bigtheorem}, that is, we present a suitable set of $8n-48$ vectors. Section 4 is devoted entirely to the proof of the upper bound in Theorem~\ref{bigtheorem}.

\section{Notation and auxiliary statements}\label{sec2}

Recall that, for a positive integer $k$, $[k]$ denotes the set $\{1,\ldots,k\}$. For a vector \(v\in\{0,\pm1\}^n\), we use the standard notation $\supp(v):=\{i\in[n]:v_i\ne0\}$.

We introduce the following notation for the set of vectors with exactly one nonzero coordinate:
$$ P = P_n \coloneqq \bigl\{v \in \{0, \pm 1\}^n \bigm| (v, v) = 1 \bigr\}. $$
We refer to the vectors in $P$ as \textit{positions}.

For a vector $v \in \{0,\pm1\}^n$, define $p(v) \subset P$ to be the set of positions that occur in this vector, that is,
$$p(v) \coloneqq \bigl\{a \in P \bigm| (v, a) = 1 \bigr\}.$$
For a set of vectors $F \subset \{0,\pm 1\}^n$, denote
$$\ p(F) \coloneqq \bigcup_{v \in F}p(v).$$
For a set of positions $P' \subset P$, we define
$$u(P') \coloneqq \sum_{a \in P'} a,$$
where $u(P') \in \{0,\pm1\}^n$.

Let $\M$ be a family of vectors satisfying the condition of Theorem~\ref{bigtheorem}. For $u,v \in \M$, we write $u \sim v$ if and only if $(u,v) \geqslant 3$. It is straightforward to check that $\sim$ is an equivalence relation on $\M$, since the values 1 and 2 of the inner product are forbidden. Thus, all vectors in $\M$ are partitioned into groups in which all pairwise inner products are equal to 3. Denote the set of equivalence classes by $\M/\sim$. Below we classify these classes according to their structure.

\begin{definition}
    A set $S \subset \M$ is called a snake if $|S| \geqslant 2$ and there exists $h \in \{0,\pm 1\}^n$ such that
    $(h, h) = 3$ and $(v, h) = 3$ for all $v \in S$.
\end{definition}
The vector $h$ in the definition is called the head of the snake $S$. All vectors in a snake have pairwise distinct nonzero coordinates outside the head.

\begin{definition}
    A set $C \subset \M$ is called a 5-cluster if $|C| \geqslant 3$ and there exists $h \in \{0,\pm 1\}^n$ such that
    $(h, h) = 5$ and $(v, h) = 4$ for all $v \in C$.
\end{definition}

If an equivalence class consists of a single element, we call it a \textit{single vector}. In Section~\ref{ss31} we prove that all equivalence classes in $\M/\sim$ are either single vectors, snakes, or 5-clusters.


We also need several statements and definitions concerning configurations of snakes.

\begin{claim}\label{prop42}
    Let $S \in \M/\sim$ be a snake with head $h$ such that $|S| \geqslant 4$. Then $(v, h) \leqslant 0$ for all $v \in \M \setminus S$.
\end{claim}

\begin{proof}
    Suppose, to the contrary, that there exists a vector $ v \in \M \setminus S$ such that $(v,h) \geqslant 1$. Since $|S| \geqslant 4$, there exists $w \in S$ such that $(\supp(w) \cap \supp(v)) \subset \supp(h)$. In this case $(v,w)=(v,h) \geqslant 1$, which is impossible because $v \notin S$. This contradiction proves the claim.
\end{proof}

\begin{definition}
    Let $S_1, S_2$ be snakes with heads $h_1, h_2$. We say that the snakes $S_1, S_2$ are \textit{compatible} if $(h_1, h_2) \leqslant -1$.
\end{definition}

\begin{definition}
    Let $S_1, S_2, S_3, S_4 \in \M/\sim$ be four distinct snakes. We say that these snakes form a serpentarium if they are pairwise compatible and $|S_i| \geqslant 4$ for every $i \in [4]$.
\end{definition}

\section{An example of an extremal configuration}
In this section we give an example of a configuration for which the bound in Theorem~\ref{bigtheorem} is attained. Let $n>6$, and let $P'$ be the set of vectors in $\mathbb R^n$ with a 1 in exactly one of the last $n-6$ coordinates and zeros in all other coordinates.

Denote
$$
H_1 = \begin{pmatrix} g_1 \\ g_2 \\ g_3 \\ g_4 \end{pmatrix} =
\begin{pmatrix}
1 & 1 & 1 & 0 & \dots & 0 \\
1 & -1 & -1 & 0 & \dots & 0 \\
-1 & 1 & -1 & 0 & \dots & 0 \\
-1 & -1 & 1 & 0 & \dots & 0
\end{pmatrix} \in M_{4 \times n}(\mathbb R),
$$

$$
H_2 = \begin{pmatrix} h_1 \\ h_2 \\ h_3 \\ h_4 \end{pmatrix} =
\begin{pmatrix}
0 & 0 & 0 & 1 & 1 & 1 & 0 & \dots & 0 \\
0 & 0 & 0 & 1 & -1 & -1 & 0 & \dots & 0 \\
0 & 0 & 0 & -1 & 1 & -1 & 0 & \dots & 0 \\
0 & 0 & 0 & -1 & -1 & 1 & 0 & \dots & 0
\end{pmatrix} \in M_{4 \times n}(\mathbb R).
$$
Note that $(h_i,h_j) = (g_i, g_j) = -1$ for $i\not=j \in [4]$. The families $\mathcal G = \{g_1, g_2, g_3, g_4\}$ and $\mathcal H =  \{h_1, h_2, h_3, h_4\}$ serve as two sets of heads of pairwise compatible snakes. Put
$$
\M = \{g + a \bigm| g \in \mathcal G, a \in P'
\} \cup \{h - a \bigm| h \in \mathcal H, a \in P'
\}.
$$
Clearly, every vector in $\M$ has 4 nonzero coordinates from the set $\{-1,1\}$, and $|\M| = 8|P'| = 8n-48$. It remains to verify that the vectors in $\M$ satisfy the conditions on pairwise inner products.

We call the vectors in $\{g + a \bigm| g \in \mathcal G, a \in P'\}$ vectors of the first type, and the vectors in $\{h - a \bigm| h \in \mathcal H, a \in P'\}$ vectors of the second type. Choose arbitrary distinct vectors $u,v \in \M$. If $u$ and $v$ belong to the same snake, then $(u,v)=3$. If $u$ and $v$ belong to different snakes but are of the same type, then $(u,v) \in \{-1,0\}$. If $u$ and $v$ are of different types, then again $(u,v) \in \{-1,0\}$. Thus, the conditions on inner products are satisfied for all pairs of vectors in $\M$, and hence the example satisfies the required condition.

\section{The upper bound in Theorem~\ref{bigtheorem}}

\subsection{Types of equivalence classes}\label{ss31}

We prove a lemma that shows that $\M/\sim$ contains only the structures described in Section~\ref{sec2}.

\begin{lemma}\label{lem22}
    Let $F \subset \M$ be an equivalence class in $\M$ with respect to $\sim$. Then exactly one of the following three conditions holds:\\
    (1) $|F| = 1$;\\
    (2) $F$ is a snake;\\
    (3) $F$ is a 5-cluster.
\end{lemma}

\begin{proof}

The case $|F|=1$ is described in condition (1). If $|F|=2$, then $F=\{v_1, v_2\}$ is a snake with head $h = u(p(v_1)\cap p(v_2))$. Now let $|F| \geqslant 3$. Consider arbitrary distinct vectors $v_1,v_2\in F$. Put $h = u(p(v_1)\cap p(v_2))$. Then $(h,h)=3$ and $(v_1,h)=(v_2,h)=3$. For any $v \in F \setminus \{v_1\}$ we have $(v,v_1)=3$, and therefore $(v,h)\geqslant 2$. There are two cases.

\emph{Case 1.} Suppose that there exists $v_3 \in F$ such that $(v_3,h)=3$. In this case $v_1,v_2,v_3$ agree on the nonzero coordinates in $\supp(h)$. If $(v,h)=3$ for all other $v \in F \setminus \{v_1,v_2,v_3\}$, then $F$ is a snake. Suppose that for some $v\in F$ we have $(v,h)=2$. Then, for each $i \in [3]$, the vector $v$ coincides with $v_i$ on exactly two coordinates of $\supp(h)$ and on one coordinate outside $\supp(h)$. But the nonzero coordinates of $v_1,v_2,v_3$ outside $\supp(h)$ are pairwise distinct, so $v$ cannot have a common additional coordinate with all three vectors $v_1,v_2,v_3$, a contradiction.

\emph{Case 2.} Now suppose that $(v,h)=2$ for every $v \in F \setminus \{v_1,v_2\}$. Choose an arbitrary vector $v_3 \in F \setminus \{v_1,v_2\}$. Then $v_3$ has a zero at one coordinate of $\supp(h)$ and shares one additional coordinate with $v_1$ and one distinct additional coordinate with $v_2$. Put $g = u(p(v_1)\cup p(v_2)\cup p(v_3))$. Then $(g,g)=5$ and $(v_i,g)=4$ for $i \in [3]$. For any $v \in F$ we have $(v,g) \geqslant 3$, since $(v,v_i) \geqslant 3$ for every $i \in [3]$. Suppose that for some $v\in F$ we have $(v,g) = 3$. Denote $a_i = u(p(v_i)\setminus p(h))$ for $i \in [2]$. Then $g-h=a_1+a_2$. From $(v,g)=3$ and $(v,h)= 2$ we get $(v,g-h)=1$. Without loss of generality, assume that $(v,a_1)=1$ and $(v,a_2)=0$. Then $(v,v_2)=(v,h)+(v,a_2)=2$, a contradiction. Hence $(v,g)=4$ for all $v\in F$, that is, $F$ is a 5-cluster.

\end{proof}

We now have a classification of equivalence classes. Let
$S_1,\ldots,S_k \subset \M$ be all equivalence classes that are snakes, and let $C_1,\ldots,C_l \subset \M$ be all equivalence classes that are 5-clusters. Denote the set of single vectors by $\T = \{ v \in \M \bigm| \nexists u \in \M: (u,v) = 3 \}$. Define $m = |\T|$. Thus, the total number of equivalence classes is equal to $k+m+l$.

For each such class $F$, we need to count the number of positions used by the vectors in $F$. We shall use the following simple observation.

\begin{claim}\label{cla23}
    Let $F \subset \M$ be an equivalence class with respect to $\sim$. Then

    {\bf (1)} if $F$ is a snake or a single vector, then $|p(F)| = |F| + 3$;

    {\bf (2)} if $F$ is a 5-cluster, then $|p(F)| \geqslant |F|$ and $|p(F)| = 5$.

    In particular, $|p(F)| \geqslant |F|$ for every equivalence class $F$.
\end{claim}

\subsection{Vectors with a fixed nonzero coordinate}

For each position $a \in P$, let $f(a)$ denote the number of equivalence classes in which the position $a$ is used. Formally,
$$f(a) \coloneqq |\{K \in \M/\sim \bigm| a \in p(K)\}|.$$

By double counting the incidence of a position with an equivalence class, we have
\begin{equation}\label{dc}
    \sum_{a \in P} f(a) = \sum_{K \in \M/\sim} |p(K)|.
\end{equation}

In this section we prove upper bounds on the function $f$.

\begin{lemma}\label{lem31}
    For every $a \in P$ we have $f(a) \leqslant 4$.
\end{lemma}

\begin{proof}

Set $s=f(a)$. Let $w_1, \dots, w_s$ be vectors belonging to distinct equivalence classes such that $a \in p(w_i)$ for every $i \in [s]$.

Put $v_i = w_i - a$. For any distinct $i,j \in [s]$ we have $(w_i, w_j) \leqslant 0$, hence $(v_i, v_j) \leqslant -1$.

We obtain the following chain:

\begin{equation}\label{eq2}
    0 \leqslant \left(\sum_{j = 1}^s v_j, \sum_{j = 1}^s v_j\right) = \sum_{j = 1}^s (v_j, v_j) + \sum_{\substack{i, j \in [s] \\ i \not= j}} (v_{i}, v_{j}) \leqslant 3s - s(s-1) = s(4-s).
\end{equation}
The condition $s(4-s) \geqslant 0$ implies $s \leqslant 4$, as required.

\end{proof}

Next we show that the bound $f(a) \leqslant 4$ can be improved if at least two vectors using the position $a$ occur in the same equivalence class.

\begin{lemma}\label{lemheads}
    Let $a \in P$, and suppose that there exist $w_1, w_2 \in \M$ such that $(w_1, w_2) = 3$ and $a \in p(w_1) \cap p(w_2)$. Then $f(a) \leqslant 3$.
\end{lemma}

\begin{proof}
    We use the notation from the proof of Lemma~\ref{lem31}. We show that $s$ cannot be $4$. If $s=4$, then all inequalities in the chain~\eqref{eq2} become equalities. Thus, we obtain $\sum_{i=1}^4 v_i = 0$.

    By assumption, one of the equivalence classes contains at least two vectors using the position $a$. Without loss of generality, there exists $w_1' \ne w_1$ in the same equivalence class as $w_1$ such that $a \in p(w_1')$. Denote $v_1' = w_1'-a$. Applying the same argument to this vector, we get $v_1'+v_2+v_3+v_4=0$, whence $v_1=v_1'$. Therefore $w_1$ and $w_1'$ coincide, a contradiction.

\end{proof}

We record several consequences of Lemma~\ref{lemheads}.

\begin{claim}\label{cla33}
    Let $C$ be a 5-cluster. Then for every position $a \in p(C)$ we have $f(a) \leqslant 3$.
\end{claim}

\begin{proof}
    Consider an arbitrary position $a \in p(C)$. In a 5-cluster $C$, at most one vector does not contain the position $a$, since each vector in the 5-cluster uses 4 positions out of 5 fixed positions. By definition, $C$ contains at least 3 vectors, so there exist two vectors $w_1,w_2 \in C$ such that $a \in p(w_1) \cap p(w_2)$. Therefore Lemma~\ref{lemheads} gives $f(a) \leqslant 3$.
\end{proof}

\begin{lemma}\label{lemsuphead}
    Fix a position $a \in P$. Then\\
    (1) Suppose that there exists a snake $S_i$ with head $h$ such that $a \in p(h)$. Then $f(a) \leqslant 3$.\\
    (2) Suppose that snakes $S_i, S_j$ with heads $h_i$ and $h_j$, respectively, are compatible and $|S_i| \geqslant 4$, $|S_j| \geqslant 4$. If their heads have a common position $a \in p(h_i) \cap p(h_j)$, then $f(a)=2$.
\end{lemma}

\begin{proof}
    (1) Consider distinct $w_1,w_2 \in S_i$. Clearly,
    $(w_1,w_2)=3$ and $a \in p(h)=p(w_1)\cap p(w_2)$. Hence Lemma~\ref{lemheads} gives $f(a) \leqslant 3$.

    (2) By assumption, $S_i$ and $S_j$ are compatible, so $(h_i,h_j) \leqslant -1$. Since the heads $h_i$ and $h_j$ have the common position $a$, their other two positions must be opposite; that is, $h_i+h_j=2a$. By Claim~\ref{prop42}, for every vector $w \in \M \setminus (S_i \cup S_j)$ we have $(w,h_i) \leqslant 0$ and $(w,h_j) \leqslant 0$. Hence $(w,2a)=(w,h_i+h_j) \leqslant 0$, that is, $a \notin p(w)$ for every $w \in \M \setminus (S_i \cup S_j)$. Thus $f(a)=2$.

\end{proof}

\subsection{The number of equivalence classes}

In this section we proof upper bound for the number of equivalence classes of $\M$.

\begin{lemma}\label{mainlem}
    If $|\M| \geqslant 8n-48$, then $k+l+m<30$.
\end{lemma}

\begin{proof}
    Let $P' \coloneqq \bigcup_{i=1}^l p(C_i)$ be the set of positions that occur in 5-clusters.

    By Claim~\ref{cla33}, for every $a \in P'$ we have $f(a) \leqslant 3$, and therefore
    \[
    \sum_{i=1}^l |p(C_i)| = 5l \leqslant \sum_{a\in P'} f(a) \leqslant 3|P'| \implies |P'| \geqslant \frac{5}{3}l.
    \]

    Combining this bound with Lemma~\ref{lem31}, we get
    \[
    \sum_{a \in P} f(a) = \sum_{a\in P'} f(a) + \sum_{a\in P \setminus P'} f(a)  \leqslant 4|P| - |P'| \leqslant 8n - \frac{5}{3}l.
    \]

    Using Claim~\ref{cla23}, we obtain the following chain:
    \begin{align*}
     8n - \frac{5}{3}l \geqslant    \sum_{a \in P} f(a) = \sum_{F \in \M/\sim} |p(F)| = \sum_{i=1}^k |p(S_i)| + \sum_{i=1}^l |p(C_i)| + \sum_{v \in \T} |p(v)| \\
    \geqslant  |\M| + 3(k+m).
    \end{align*}

    Thus $|\M|+3(k+m) \leqslant 8n-\frac{5}{3}l$. Since $|\M| \geqslant 8n-48$, it follows that $k+m+\frac{5}{9}l \leqslant 16$, and hence $k+m+l \leqslant 16 \cdot \frac{9}{5}<30$. Lemma~\ref{mainlem} is proved.
\end{proof}

\subsection{Configurations of snakes}


In this section we show that, for sufficiently large $n$, two serpentaria can be found in an extremal example.

\begin{lemma}\label{thm3}
    There exists $n_0$ such that, if $n>n_0$ and $|\M| \geqslant 8n-48$, then there are two disjoint serpentaria.
\end{lemma}

\begin{proof}
    For a vector $v$ we write $p_{pm}(v)=p(v)\cup\{-a:a\in p(v)\}$. Define $p_{pm}(F)$ analogously for a set of vectors $F$.
    Define the sets $H_1,H_2,H_3,H_4 \subset P$ as follows:
    \[
    H_1 \coloneqq p_{pm}(\{h \bigm| \exists i \in [k], h   \text{ is the head of } S_i\}),
    \]
    \[
    H_2 \coloneqq \bigcup_{i\in [l]} p_{pm}(C_i),
    \]
    \[
    H_3 \coloneqq \bigcup_{v \in \T} p_{pm}(v),
    \]
    \[
    H_4 \coloneqq \bigcup_{i \in [k], |S_i| \leqslant 3} p_{pm}(S_i),
    \]
    \[
    H \coloneqq P \setminus (H_1 \cup H_2 \cup H_3 \cup H_4).
    \]
    From the definitions it follows that $|H_1| \leqslant 6k$, $|H_2| \leqslant 10l$, $|H_3| \leqslant 8m$, and $|H_4| \leqslant 12k$. By Lemma~\ref{mainlem}, we have $k+l+m<30$. Therefore
    $$|H_1|+|H_2|+|H_3|+|H_4| \leqslant 18(k+l+m) \leqslant 18\cdot29=522.$$
    Hence
    $$|H| \geqslant |P|-|H_1|-|H_2|-|H_3|-|H_4| \geqslant 2n-522.$$

    We prove that for sufficiently large $n$ there exists a position $a \in H$ such that $f(a)=f(-a)=4$. We argue by contradiction. Note that if $a \in H$, then $-a \in H$. Let $n>309$. Then $|H|>2\cdot309-522=96$. We have
    \[
    \sum_{a\in P} f(a) = \sum_{a\in H} f(a) + \sum_{a\in P \setminus H} f(a)
    \leqslant
    4|P| - \frac{1}{2}|H| < 8n-48 \implies |\M|<8n-48.
    \]
    This is a contradiction. Now fix $a\in H$ for which $f(a)=f(-a)=4$. Since $a\notin H_2\cup H_3\cup H_4$, every class that uses $a$ is a snake of cardinality at least 4. Since $a\notin H_1$, the position $a$ does not lie in the head of such a snake. We show that the snakes containing the position $a$ are pairwise compatible. Consider any two distinct such snakes $S_1,S_2$ with heads $h_1,h_2$, respectively. Then
    $$h_1+a\in S_1, \quad h_2+a\in S_2 \implies (h_1+a,h_2+a)\leqslant 0 \implies (h_1,h_2)\leqslant -1,$$
    as required. Thus, the snakes containing the position $a$ form a serpentarium. The same argument applies to $-a$. It remains only to observe that these serpentaria cannot have common snakes, since a snake cannot use both positions $a$ and $-a$ at the same time. Hence we have found two disjoint serpentaria, and the lemma is proved.
\end{proof}

\subsection{Completion of the proof}

\begin{lemma}\label{lemserp}
    Fix a serpentarium, and let $\mathcal H$ be the set of heads of its snakes. For $a \in P$, denote $\mathcal H_a \coloneqq \{h \in \mathcal H: a \in p(h)\}$. Then the following properties hold.

    (1) For any two distinct $h_1 \not= h_2 \in \mathcal H$, we have $(h_1,h_2)=-1$.

    (2) For all $a \in P$, we have $|\mathcal H_a|=|\mathcal H_{-a}|$.

    (3) For every $h \in \mathcal H$ and every vector $v \in \M$ not belonging to any snake in this serpentarium, we have $(v,h)=0$.

    (4) Let $h_i$ be the head of a snake $S_i$ that does not belong to this serpentarium. If $|S_i|\geqslant 4$, then $(h_i,h)=0$ for every $h\in\mathcal H$.
\end{lemma}

\begin{proof}
    {\bf (1)} By the definition of a serpentarium, for any $h_1 \not= h_2 \in \mathcal H$ we can say that $(h_1,h_2)\leqslant -1$. We have
    \begin{equation}\label{eqheads}
        0 \leqslant \left(\sum_{h \in \mathcal H} h, \sum_{h \in \mathcal H} h\right) = 3\cdot4 + \sum_{h_1 \not=h_2 \in \mathcal H} (h_1,h_2) \leqslant 12-12=0.
    \end{equation}
    It follows that
    $$\sum_{h_1 \not=h_2 \in \mathcal H} (h_1,h_2)=-12.$$
    This is possible only if $(h_1,h_2)=-1$ for all $h_1 \not= h_2\in\mathcal H$.

    {\bf (2)} All inequalities in~\eqref{eqheads} must be equalities, and hence $\sum_{h\in\mathcal H} h=0$. It follows that the positions $a$ and $-a$ occur in the same number of heads from $\mathcal H$, as required.

    {\bf (3)} Consider an arbitrary vector $v\in\mathcal M$ that does not belong to any snake in this serpentarium. By Claim~\ref{prop42}, we have $(v,h)\leqslant 0$ for every $h\in\mathcal H$. Observe that
    $$
    \left(v,\sum_{h \in \mathcal H} h \right) = \sum_{h \in \mathcal H} (v,h) = \sum_{a \in p(v)} \sum_{h \in \mathcal H}  (a,h) = 0.
    $$
    Therefore $(v,h)$ must be zero for every $h\in\mathcal H$.

    {\bf (4)} Consider an arbitrary head $h\in\mathcal H$. Suppose that $(h_i,h)\geqslant 1$. On the other hand, for every vector $w\in S_i$, Claim~\ref{prop42} gives $(w,h)\leqslant 0$. This means that $\supp(w)$ must intersect $\supp(h)\setminus\supp(h_i)$. This contradicts the assumption $|S_i|\geqslant 4$. Thus $(h_i,h)\leqslant 0$.

    As in part (3), we have
    $$
    \left(h_i,\sum_{h \in \mathcal H} h \right) = \sum_{h \in \mathcal H} (h_i,h) = \sum_{a \in p(h_i)} \sum_{h \in \mathcal H}  (a,h) = 0,
    $$
    so $(h_i,h)$ must be zero for every $h\in\mathcal H$.
\end{proof}

By Lemma~\ref{thm3}, the set $\M$ already contains two disjoint serpentaria; denote them by $\mathcal S_1$ and $\mathcal S_2$. Without loss of generality, assume that $\mathcal S_1=\{S_1,S_2,S_3,S_4\}$ and $\mathcal S_2=\{S_5,S_6,S_7,S_8\}$. This already implies $k\geqslant 8$. Using equality~\eqref{dc}, we get
    \begin{align}\label{eq4}
        \sum_{a \in P} f(a) = \sum_{F \in \M/\sim}|p(F)| \geqslant \sum_{F \in \M/\sim}|F| + 3k \geqslant \sum_{F \in \M/\sim}|F| + 24.
    \end{align}

In the serpentaria $\mathcal S_1$ and $\mathcal S_2$, denote the sets of heads by $\mathcal H_1$ and $\mathcal H_2$, respectively, and also put $\mathcal H=\mathcal H_1\cup\mathcal H_2$.

Next we bound the sum $\sum_{a\in P}f(a)$. Let $P_1$ be the set of positions that are used in exactly one head from $\mathcal H$, and let $P_2$ be the set of positions that are used in at least two heads from $\mathcal H$.


\begin{lemma}\label{lemserpent}
    (1) For every position $a \in P_1$, inequality $f(a)\leqslant 3$ holds.

    (2) For every position $a \in P_2$, equality $f(a)=2$ holds.
\end{lemma}

\begin{proof}

    {\bf (1)} This follows from part (1) of Lemma~\ref{lemsuphead}.

    {\bf (2)} If the position $a$ is used in two heads from one serpentarium, then the statement follows from Lemma~\ref{lemsuphead}, since any two snakes of the same serpentarium are compatible. Now suppose that in each serpentarium there is exactly one head using $a$. Without loss of generality, $a\in p(h_1)$ and $a\in p(h_5)$, where $h_1$ and $h_5$ are the heads of the snakes $S_1$ and $S_5$, respectively.

    First note that no vector belonging to the other snakes of these two serpentaria can use the position $a$. Indeed, suppose, without loss of generality, that a vector $v$ belongs to a snake of the first serpentarium with head $h\ne h_1$ and $a\in p(v)$. Since $a\notin p(h)$, we have $v=h+a$. By Lemma~\ref{lemserp}, $(v,h_5)=0$ and $(h,h_5)=0$. Hence
    $$(v,h_5)=(h,h_5)+(a,h_5)=1,$$
    a contradiction.

    Now consider an arbitrary vector $v\in\M$ that does not belong to the snakes from $\mathcal S_1\cup\mathcal S_2$. Our goal is to show that $a \notin p(v)$.

    By part (2) of Lemma~\ref{lemserp}, in the first serpentarium there is a head $h_2$ such that $-a\in p(h_2)$, and in the second serpentarium there is a head $h_6$ such that $-a\in p(h_6)$. Consider the vector $h_1+h_5-h_2-h_6$. By Lemma~\ref{lemserp}, we know that heads from different serpentaria are orthogonal, and therefore
    $$(h_1+h_5-h_2-h_6, h_1+h_5-h_2-h_6)=4\cdot3-2(h_1,h_2)-2(h_5,h_6)=16.$$
    On the other hand, this vector has coordinate value 4 in the coordinate corresponding to $a$, hence $h_1+h_5-h_2-h_6=4a$.

    It remains to note that, by part (3) of Lemma~\ref{lemserp}, the vector $v$ is orthogonal to each of the vectors $h_1,h_2,h_5,h_6$, so
    $$(v,4a)=(v,h_1+h_5-h_2-h_6)=0.$$
    Hence $a\notin p(v)$. Thus, only the classes $S_1$ and $S_5$ use the position $a$. Therefore $f(a)=2$.

\end{proof}

Now we know that $p(\mathcal H_1 \sqcup \mathcal H_2)=P_1\sqcup P_2$. By double counting, $|P_1|+2|P_2|=8\cdot3=24$. Applying Lemma~\ref{lemserpent}, we obtain

\begin{align}\label{eq5}
    \sum_{a \in P} f(a) = \sum_{a \in P \setminus (P_1 \cup P_2)} f(a) + \sum_{a \in P_1} f(a) + \sum_{a \in P_2} f(a) \leqslant 4|P| - |P_1| - 2|P_2| = 8n-24.
\end{align}

Combining estimates~\eqref{eq4} and~\eqref{eq5}, we get

\[
|\M| = \sum_{F \in \M/\sim}|F| \leqslant \sum_{a \in P} f(a) - 24 \leqslant 8n-48.
\]
The upper bound in Theorem~\ref{bigtheorem} is proved.

\paragraph{Acknowledgements.}
The authors are grateful to A. M. Raigorodskii for productive discussions and to D. D. Cherkashin for useful comments.

\bibliographystyle{plain}
\bibliography{bibliography}

\end{document}